\documentclass[12pt]{amsart}

\usepackage{amsmath,amssymb,amsthm, enumerate}

\newtheorem{theorem}{Theorem}[section]

\newtheorem{lemma}[theorem]{Lemma}
\newtheorem{proposition}[theorem]{Proposition}

\theoremstyle{definition}

\newtheorem{example}[theorem]{Example}

\theoremstyle{remark}
\newtheorem{remark}[theorem]{Remark}

\newcommand{\bC}{\mathbb C}
\newcommand{\bR}{\mathbb R}
\newcommand{\bS}{\mathbb S}

\newcommand{\aut}{\mathrm{Aut}}

\DeclareMathOperator{\im}{Im}
\DeclareMathOperator{\re}{Re}

\DeclareMathOperator{\esstype}{ess\, type}
\DeclareMathOperator{\mult}{mult}

\begin{document}
\title[Holomorphic mappings between pseudoellipsoids]{Holomorphic mappings between pseudoellipsoids in different dimensions}
\author{Peter Ebenfelt}
\address{Department of Mathematics, University of California at San Diego, La Jolla, CA 92093-0112}
\email{pebenfel@math.ucsd.edu,}
\author{Duong Ngoc Son}
\address{Department of Mathematics, University of California at Irvine, Irvine, CA 92697-3875}
\email{snduong@math.uci.edu}
\date{\today}
\thanks{The first author was partly supporting by the NSF grant DMS-1001322.
The second author acknowledges a scholarship from the Vietnam Education Foundation.}
\begin{abstract} We give a necessary and sufficient condition for the existence of nondegenerate
holomorphic mappings between pseudoellipsoidal real hypersurfaces, and  provide an explicit
parametrization for the collection of all such mappings (in the situations where they exist).
\end{abstract}
\maketitle

\section{Introduction}

In recent years, much effort has been devoted to the ambitious program
of classifying local holomorphic mappings $H$ (or more generally CR mappings)
sending a given real hypersurface $M\subset \bC^{n+1}$ into a model
hypersurface $M'\subset \bC^{N+1}$. In the strictly pseudoconvex case,
the natural model is the sphere $M'=\bS^{2N+1}\subset \bC^{N+1}$ (and,
more generally, in the Levi nondegenerate case the model
is the hyperquadric of signature $l$). Of particular interest is the case
where the source manifold $M$ is the model itself.
There is a large body of work studying local holomorphic mappings $H$
sending a piece of the sphere $\bS^{2n+1}$ into $\bS^{2N+1}$. The reader
is referred to e.g.\ \cite{Webster79}, \cite{Faran86}, \cite{Forstneric86},
\cite{DAngelo88}, \cite{Dangelo91}, \cite{Huang99}, \cite{HuangJi01},
\cite{HuangJiXu06}, \cite{JPDLeblPeters07}, \cite{HuangJiYin12} and the
references therein; see also \cite{EHZ04}, \cite{EHZ05} for the case of
general strictly pseudoconvex source manifolds $M$. The classification of such  holomorphic mappings in {\it low codimensions} $N-n$ is completely
understood due to the works \cite{Faran86}, \cite{HuangJi01}, \cite{HuangJiXu06},
\cite{HuangJiYin12}. In particular, it was shown in \cite{Faran86} that any
local holomorphic mapping $H$ sending a piece of $\bS^{2n+1}$ into $\bS^{2N+1}$
with $N-n<n$ is necessarily of the form $H=T\circ L$, where $L$ denotes the
standard linear embedding of $\bS^{2n+1}$ into an $(n+1)$-dimensional complex
subspace section of $\bS^{2N+1}$ and $T$ is an automorphism of $\bS^{2N+1}$.
(This is often referred to as {\it rigidity}.)
In this note, we shall consider local (non-constant) holomorphic mappings
sending a pseudoellipsoid in $\bC^{n+1}$ into another pseudoellipsoid in
$\bC^{N+1}$ in low codimension.
A pseudoellipsoid in $\bC^{N+1}$ is a compact, algebraic real hypersurface
of the form
\begin{equation}\label{psellipse}
 E^N_q = \{(z,w)\in \bC^n\times \bC \colon \langle z^q,\bar{z}^q \rangle + |w|^2 =1\}
\end{equation}
where $q = (q_1,q_2,\dots q_N)$ is an $N$-tuple of integers with $q_j\ge 1$,
$z^q = (z_1^{q_1},\dots ,z_N^{q_N})$ and
$\langle \cdot,\cdot\rangle $ denotes the standard hermitian form in $\bC^N$,
$$
\langle u,v\rangle:=\sum_{j=1}^{N}u_jv_j,\quad u,v\in\bC^N.
$$
We note that $E^N_q$ is weakly pseudoconvex along those coordinate planes $z_j=0$
for which $q_j\geq 2$, and in particular at the point $p_0:=(0,\ldots,0,1)\in E^N_q$
(unless all $q_j=1$ and $E^N_q$ is the sphere). The pseudoellipsoids are natural models
(albeit not homogeneous in general, of course) for certain classes of weakly pseudoconvex
hypersurfaces. The domains that they bound are the only ones, up to biholomorphic
equivalence, with noncompact automorphism groups in a fairly general class of
smoothly bounded pseudoconvex domains \cite{BedfordPinchuk91} (as with the case of the ball in the strictly pseudoconvex category \cite{Wong77}); see also \cite{IsaevKrantz99}. Biholomorphic equivalence of pseudoellipsoids, their automorphism groups, as well as the existence
and geometric properties of (non-constant) local holomorphic mappings between
pseudoellipsoids in the equidimensional case (i.e.\ the source and target are both
hypersurfaces in $\bC^{n+1}$) have been investigated in e.g.\ \cite{Landucci84},
\cite{DiniSelvaggi91}, \cite{DiniSelvaggi97}, \cite{LanducciSpiro09},
\cite{MontiMorbidelli12}. In particular,
the following result follows from these works:
\medskip

\noindent
{\bf Theorem 0.} {\it Let $p=(p_1,\ldots, p_n)$ and
$q=(q_1,\ldots, q_n)$ be $n$-tuples of positive integers. Then there exists a
non-constant local holomorphic mapping $H\colon (\bC^{n+1},p_0)\to (\bC^{n+1},p_0)$,
with $p_0:=(0,\ldots, 0,1)\in \bC^{n+1}$, sending $E^n_p$ into $E^n_q$ if and only
if there exists a permutation
$\sigma\colon \{1,\ldots,n\}\to \{1,\ldots,n\}$ such that $q_k|p_{\sigma(k)}$
for all $k=1,\ldots, n$. 

Moreover, if such mappings $H$ exist, then the collection of all such $H$ can be described as follows:
$$
H(z,w)=T\circ (z_{\sigma(1)}^{p_{\sigma(1)}/q_1},\ldots,z_{\sigma(n)}^{p_{\sigma(n)}/q_n}, w),
$$
where $\sigma$ ranges over all permutations $\{1,\ldots,n\}\to \{1,\ldots,n\}$ such that $q_k|p_{\sigma(k)}$, for $k=1,\ldots, n$, and $T$ over the automorphisms of $E^n_q$.}
\medskip

\noindent
A complete and explicit description of the automorphism group of $E^n_q$ also follows from the works mentioned above (see e.g.\ \cite{Landucci84}). For the readers convenience, we provide in Section \ref{s:auto} a description (or, more precisely, a decomposition into elementary mappings) of the stability group of $E^n_q$ at $p_0$, i.e.\ the group of automorphisms of $E^n_q$ preserving $p_0$.

The main purpose of this note is to extend Theorem 0 to the positive (but low) codimensional situation. We note, for reference,
that the defining equation for $E^N_q$ is plurisubharmonic and, hence, by a standard
application of the Hopf boundary point lemma it follows that any nonconstant holomorphic
mapping sending $E^n_p$ into $E^N_q$ is necessarily transversal to $E^N_q$.
In what follows, we shall only consider holomorphic mappings that are transversal to their target manifolds.

It is convenient to note that $E^N_q$ minus a point is biholomorphically equivalent (via a linear fractional transformation) to the real algebraic hypersurface given by
\begin{equation}\label{e101}
 P_q^N = \{(z,w) \in \bC^N \times \bC: \im w = \langle z^q, \bar z^q\rangle \},
\end{equation}
with the point $p_0:=(0,\ldots,0,1)$ on $E^N_q$ corresponding to the origin on $P^N_q$.

Our main result is a necessary and sufficient condition for
the existence of local holomorphic mappings $H\colon (\bC^{n+1},0)\to (\bC^{N+1},0)$
sending $P^n_{p}$ transversally into $P^N_{q}$, as well as a description of the collection of all such mapping (when they exist). The latter description is also given in an explicit formula in Theorem \ref{explicitformula} below.

\begin{theorem}\label{Main1} Consider $P^n_{p}\subset \bC^{n+1}$ and
$P^N_{q}\subset \bC^{N+1}$, where $p=(p_1,\ldots, p_n)$ and $q=(q_1,\ldots, q_N)$
are an $n$-tuple and $N$-tuple, respectively, of positive integers, the latter arranged such that
$q_1=\ldots=q_s=1$ and $q_{k}\geq 2$, $k=s+1,\ldots, N$, for some $s\geq0$. Assume that
\begin{equation}\label{codimcond}
N-n<n.
\end{equation}
The following are equivalent:
\begin{enumerate}[{\rm (i)}]
\item There exist a subset $K\subset \{s+1,\dots, N\}$ (possibly empty) and a
map $\sigma\colon K\to \{1,\ldots, n\}$ such that \(\#\sigma(K)\ge n-s\) and
\(q_k\, |\, p_{\sigma(k)} \ \text{for all}\ k\in K\).
\item There exists a local holomorphic mapping $H\colon (\bC^{n+1},0)\to (\bC^{N+1},0)$
sending $P^n_{p}$ into $P^N_{q}$, transversal to $P^N_{q}$ at $0$.
\end{enumerate}
Moreover, if {\rm (ii)} holds then the collection of all such mappings $H$ can be described as follows. Let  $\sigma$ and $K$  be as in {\rm (i)}, and $W=(u_{ij})$ an $n\times N$ matrix such that
\begin{enumerate}[{\rm (a)}]
\item $WW^*=I_{n\times n}$, and
\item for every $j\in \{s+1,\ldots, N\}$, it holds that $u_{ij}\neq 0$ if and only if $j\in K$ and $\sigma(j)=i$.
\end{enumerate}
Then, the monomial mapping $H_{\sigma,W}(z,w)=(F(z),w)$, where
\begin{equation}\label{HsigmaW}
F_j(z) =
\begin{cases}
\sum_{i=1}^nu_{ij}z_i^{p_i},\quad &j=1,\ldots, s,\\
\left( u_{\sigma(j)j}z_{\sigma(j)}^{p_{\sigma(j)}}\right)^{1/q_j},\quad &j\in K,\\
0,\quad &j\in \{s+1,\ldots, N\}\setminus K.
\end{cases}
\end{equation}
sends $P^n_{p}$ transversally into $P^N_{q}$, and any mapping $H$ as in {\rm (ii)} is of the form $H=T\circ H_{\sigma,W}$ for some $\sigma$ (and $K$), $W$,  and $T$, where $T$ is an automorphism of $P^N_{q}$ preserving the origin.
\end{theorem}

The proof of Theorem \ref{Main1} will be given in Section \ref{s:auto} below.

\begin{remark}\label{rem:intro}{\rm
In the equidimensional case $N=n$, we note that any subset $K$ and mapping $\sigma$ as in (i) in Theorem \ref{Main1} must be such that $K=\{s+1,\ldots,n\}$ and such that $\sigma$ can be extended to a permutation $\tilde \sigma$ on $\{1,\ldots,n\}$ with $q_k|p_{\tilde\sigma(k)}$ for all $k=1,\ldots, n$, and vice versa, any such permutation $\tilde \sigma$ induces a mapping $\sigma $ by taking $K:=\{s+1,\ldots, n\}$ and $\sigma:=\tilde\sigma|_K$. If we reorder the coordinates on the source side so that the permutation $\tilde \sigma$ becomes the identity, then any $n\times n$ matrix $W$ satisfying (a) and (b) has the block form
\begin{equation}\label{equidim}
W=
\begin{pmatrix} U & 0\\ 0& D
\end{pmatrix},
\end{equation}
where $U$ is a unitary $s\times s$ matrix and $D$ is a diagonal $(n-s)\times (n-s)$ matrix whose diagonal elements have modulus one. The corresponding mapping $H_{\sigma, W}$ (where now $\sigma$ is the identity on $\{s+1,\ldots,n\}$) is then of the form $H_{\sigma, W}:=T_W\circ H_0$, where 
$$H_0(z,w):=(z_1^{p_1/q_1},\ldots, z_n^{p_n/q_n},w)$$ 
and $T_W$ is the automorphism of $P^n_q$ given by $$T_W(z,w)=(z'U,z''D,w)$$
with $z'=(z_1,\ldots,z_s)$ and $z''=(z_{s+1},\ldots, z_n)$. Returning to the original ordering of the source coordinates, we recover the equidimensional result stated in Theorem 0 above. We notice a redundancy in the statement of the theorem in this case; the additional mappings afforded by the choice of $W$ can be incorporated into the action of the stability group. In the general case, there may also be some redundancy in that $H_{\sigma,W}$ could equal $T\circ H_{\sigma,W'}$ for $W\neq W'$ and a suitable choice of automorphism $T$, but for $N>n$ there will be (in general) different choices of $W$ such that the corresponding mappings are not related by an automorphism of $P^N_q$. This is explained more closely in the context of an example in Section \ref{s:ex}.
}
\end{remark}

A consequence of Theorem \ref{Main1} (or, more precisely, a consequence
of Theorem \ref{explicitformula} below) is the following ``localization principle'' (c.f.\ \cite{DiniSelvaggi97}, \cite{LanducciSpiro09}).

\begin{theorem}\label{Cor:alg} Consider $P^n_{p}\subset \bC^{n+1}$ and
$P^N_{q}\subset \bC^{N+1}$, where $p=(p_1,\ldots, p_n)$ and $q=(q_1,\ldots, q_N)$
are an $n$-tuple and $N$-tuple, respectively, of positive integers,  and assume that
$N-n<n$. If a local holomorphic mapping $H\colon (\bC^{n+1},0)\to (\bC^{N+1},0)$
sends $P^n_{p}$ transversally into $P^N_{q}$, then $H$ extends as an algebraic mapping which is holomorphic in a neighborhood of $P^n_{p}$.
\end{theorem}
The (short) proof of Theorem \ref{Cor:alg} is given in Section \ref{s:expform}.
\begin{remark}
The extendability of $H$ as an algebraic map, possibly singular and multi-valued, follows from previous results due to Huang \cite{Huang94} (see also \cite{Z99}). In the setting of Theorem \ref{Cor:alg}, it follows (see the proof) that, in fact, the $H_k^{q_k}$ are rational, and are (possibly) ramified along a complex hypersurface that does not meet $P^n_p$.
\end{remark}

We shall conclude this introduction with a brief discussion of an analogous non-pseudoconvex situation.
Consider the ``positive signature'' counterparts of $P^N_{q}$, i.e.\ the ``pseudohyperboloids'' given by
\begin{equation}\label{e101-2}
 P^N_{q,\ell} = \{(z,w) \in \bC^N \times \bC: \im w = \langle z^q, \bar z^q\rangle_\ell \},
\end{equation}
where $\langle \cdot, \cdot \rangle_{\ell}$ is the standard hermitian form of signature $\ell>0$, i.e.,
$$
\langle u,v\rangle_\ell:=-\sum_{j=1}^\ell u_j v_j+\sum_{j=\ell+1}^{N}u_j v_j,\quad u,v\in\bC^N.
$$
In the Levi nondegenerate case, i.e.\ $q=(1,\ldots, 1)$, the pseudohyperboloid
$P^N_{q,\ell}$ coincides with the standard hyperquadric $Q^N_\ell$ of signature $\ell$. For a local holomorphic mapping $H\colon (\bC^{n+1},0)\to (\bC^{N+1},0)$ sending $Q^n_\ell$ into $Q^N_\ell$, with $\ell>0$, it is known by \cite{BH05} that $H=T\circ L$, where $L$ is a linear embedding of $Q^n_\ell$ into $Q^N_\ell$ and $T$ an automorphism of $Q^N_\ell$, {\it regardless} of the codimension $N-n$ (i.e.\ {\it super-rigidity} holds, in stark contrast to the pseudoconvex case, as in \cite{Faran86}, where $N-n<n$ is necessary for rigidity to hold). By following the same arguments as in the pseudoconvex case (modulo replacing the use of the rigidity result in \cite{Faran86} by the super-rigidity result in \cite{BH05}), one obtains analogous classification results to those in Theorems \ref{Main1} and \ref{explicitformula} for local mappings $P^n_{p,\ell}\to P^N_{q,\ell}$ with $\ell>0$, the major difference being that the conditions \eqref{codimcond} and \eqref{codimcond2} in Theorems \ref{Main1} and \ref{explicitformula}, respectively, are no longer needed. In the positive signature case ($\ell>0$), one needs, however, to distinguish between the coordinates that appear with a plus sign and those that appear with a minus sign in the hermitian form $\langle\cdot,\cdot\rangle_\ell$, which has as a consequence that there are different cases to consider and the results become more cumbersome to state. The diligent reader is invited to work out the details.

\section{An explicit formula for mappings from $P_{p}^n$ into $P_{q}^N$}\label{s:expform}
In this section, we shall prove the following result, which is the main ingredient in the proof of Theorem \ref{Main1}.

\begin{theorem}\label{explicitformula} Consider $P^n_{p}\subset \bC^{n+1}$ and
$P^N_{q}\subset \bC^{N+1}$, where $p=(p_1,\ldots, p_n)$ and $q=(q_1,\ldots, q_N)$
are an $n$-tuple and $N$-tuple, respectively, of positive integers, the latter arranged such that
$q_1=\ldots=q_s=1$ and $q_{k}\geq 2$, $k=s+1,\ldots, N$, for some $s\geq0$. Assume that
\begin{equation}\label{codimcond2}
N-n<n.
\end{equation}
If  $H\colon (\bC^{n+1},0)\to (\bC^{N+1},0)$ is a local holomorphic mapping
sending $P^n_{p}$ into $P^N_{q}$, transversal to $P^N_{q}$ at $0$, then there exists
\begin{enumerate}[{\rm (A)}]
\item a subset $K\subset \{s+1,\dots, N\}$ (possibly empty) and a
map $\sigma\colon K\to \{1,\ldots, n\}$ such that \(\#\sigma(K)\ge n-s\) and
\(q_k\, |\, p_{\sigma(k)} \ \text{for all}\ k\in K\),
\item a unitary $N\times N$ matrix $U=(u_{ij})$ with the property that, for $j\in\{s+1,\ldots, N\}$, it holds that $u_{ij}\neq 0$ if and only if $j\in K$ and $\sigma(j)=i$, and
\item $r\in \bR$, $\lambda>0$, and
$b = (b', b'') \in \bC^{n} \times \bC^{N-n}$ with the property that $c:=bU$ satisfies $c_k=0$ for $k=s+1,\ldots, N$,
\end{enumerate}
such that $H$ takes the following form:
\begin{equation}\label{list}
H_j(z,w)=
\begin{cases}
 \lambda\left(\sum_{i = 1}^n u_{ij} z_i^{p_i} + c_j w\right)/\delta(z,w), \quad &j=1,2,\dots s;\\
 \left(\lambda u_{\sigma(j)j} z_{\sigma(j)}^{p_{\sigma(j)}}\right)^{1/q_j} /\delta(z,w)^{1/q_j}, &j \in K;\\
 0, &j \in \{s+1,\dots N\} \setminus K;\\
 \lambda^2 w/\delta(z,w), &j=N+1,
\end{cases}
\end{equation}
where $\delta(z,w) := 1 - 2i\langle z^p, \bar b' \rangle - (r+i\langle b,\bar b\rangle) w$, and $z^p := (z_1^{p_1},\dots z_n^{p_n})$.

Furthermore,  if $s=0$ then $b=0$, and if $N=n$ then $\sigma$ can be extended to a permutation $\tilde \sigma$ on $\{1,\ldots,n\}$ with $q_k|p_{\tilde\sigma(k)}$ for all $k=1,\ldots, n$.
\end{theorem}

\begin{proof} We introduce the map
$\tilde \phi_q(\tilde z,\tilde w) = (\tilde z_1^{q_1}, \dots \tilde z_N^{q_N} , \tilde w)$,
which is a holomorphic mapping sending $P_q^N$ into the (Heisenberg) sphere
$\mathbb H^N:=P^N_{(1,\ldots,1)}$. Thus, $\tilde \phi_q \circ H$ is a non-constant mapping
from a neighborhood of $0$ in $P_p^n$ into $\mathbb H^N$. We also introduce
$\phi_p(z,w) = (z_1^{p_1},\dots z_n^{p_n}, w)$, which is a
holomorphic mapping sending $P_p^n$ into $\mathbb H^n$. Let $a\in P_p^n$ be some point near $0$ whose coordinate components do not vanish, and $U$
a neighborhood of $a$ such that $U\cap \{z_j = 0\} =\emptyset$ for all
$j=1,\ldots , n$. We can also choose $U$ small so that $\phi_p$ is biholomorphic on $U$.

Now, let $\tau\in \aut(\mathbb H^n)$ be such that $\tau(\phi_p(a)) = 0$ and
$T\in \aut(\mathbb H^N)$ such that $T(0) = \phi_q(H(a))$. Consider the
following mapping defined on $\hat U = \tau(\phi_p(U))$
\[
 \hat H = T^{-1}\circ \phi_q \circ H \circ \phi_p^{-1} \circ \tau^{-1}.
\]
Clearly, $\hat H (\hat U \cap \mathbb H^n) \subset \mathbb H^N$
and $\hat H(0) = 0$. Since $N-n<n$, we can apply the rigidity theorem in \cite{Faran86}, mentioned in the introduction, to conclude that there
is an automorphism $\hat T\in \aut(\mathbb H^N,0)$ such that
$\hat H = \hat T \circ L$ with $L(z,w) = (z,0,w)$. This implies that the
following holds on $U$ and, by analytic continuation, in any connected open set containing $a$ where $H$ is defined (in particular, in an open neighborhood of $0$):
\begin{equation}
 \phi_q \circ H = T\circ \hat T \circ L \circ \tau \circ \phi_p.
\end{equation}
Since the non-constant mapping $T\circ \hat T \circ L \circ \tau$
sends $\mathbb H^n$ into $\mathbb H^N$ and $0$ into $0$, it follows that $T\circ \hat T \circ L \circ \tau = T'\circ L$
for some $T'\in \aut(H^N,0)$, and hence:
\begin{equation}
 \phi_q \circ H = T' \circ L\circ \phi_p
\end{equation}
It follows from the explicit description of $\aut(\mathbb H^N,0)$
(see, e.g., \cite{CM74} or \cite{BER00b}) that there are $\lambda >0$, $r\in \bR$, $b \in \bC^N$
and a unitary $N\times N$-matrix $U = (u_{ij})$ (i.e.\ $UU^*=U^*U=I$) such that
\begin{align}
 H_{N+1}(z,w) &= \lambda^2 w/\delta(z,w) \label{g}\\
 H_j^{q_j}(z,w) &= \lambda\left(\sum_{i = 1}^n u_{ij} z_i^{p_i} + c_jw\right)/\delta(z,w), \quad \text{for} \quad j=1,2,\dots N, \label{f}
\end{align}
where
\begin{align}
 c_j &= \sum_{i = 1}^N u_{ij} b_i\label{c} \quad \text{for} \quad j=1,2,\dots N\\
 \delta(z,w) &=1-2i\langle z^p, b'\rangle -(r+i\langle b,\bar b\rangle)w \label{delta}, \quad b=(b',b'')\in \bC^n\times \bC^{N-n}.
\end{align}
Recall that $q_1 = \dots q_s = 1$ and $2\le q_{s+1} \le \dots \le q_N$. By setting $z=0$ in \eqref{f},
we have
\begin{equation}\label{e310}
\left( 1-(r+i\langle b,\bar b\rangle)w\right) H_j^{q_j}(0,w) = c_jw.
\end{equation}
If $c_j \ne 0$, then $H_j^{q_j}(0,w)$ divides $w$ in $\bC\{w\}$.
This is impossible if $q_j >1$. Thus, $c_j=0$ for $j=s+1,\dots N$.
Let us define
$$K=\{ k\ge s+1\ | \ u_{tk} \ne 0\ \text{for some} \ 1\le t \le n\}.$$
For $k\in K$, we claim that there is a unique $t^* \in \{1,2,\dots n\}$
such that $u_{t^*k}\neq 0$ and $u_{tk} = 0$ for all $1\le t \le n$, $t \ne t ^*$.
To prove the claim, suppose that there are
two indices $t$, say $t = 1,2$, such that $u_{kt} \ne 0$.
Setting $z_3 = \dots = z_n =w=0$ in equation \eqref{f} we would obtain
\begin{equation}\label{e311}
 (1 - b_1z_1^{p_1} - b_2z_2^{p_2}) H_k^{q_k}(z,0) = \lambda (u_{1k}z_1^{p_1} + u_{2k}z_2^{p_2}),
\end{equation}
which is impossible. Indeed, by differentiating
both sides of \eqref{e311} with respect to $z_1$ we note that
$H_k^{q_k-1}(z,0) \ | \ z_1^{p_1-1}$. The same argument with $z_1$ replaced by $z_2$ shows that
$H_k^{q_k-1}(z,0) \ | \ z_2^{p_2-1}$, which would lead to a contradiction since
$q_k-1\ge 1$ and $H_k$ is not an unit. The claim follows. We now define a map $\sigma \colon K \to \{1,2,\dots n\}$ by
$\sigma(k) = t^*$. Thus, we note that $H_k\equiv 0$ for $k\in \{s+1,\ldots, N\}\setminus K$ and
\begin{equation}
  H_k^{q_k}(z,w) = \lambda  u_{\sigma(k)k} z_{\sigma(k)}^{p_{\sigma(k)}}/\delta(z,w), \quad \text{for} \quad k\in K.
\end{equation}
From this it readily follows that $q_k \ | \ p_{\sigma(k)}$. We conclude that
\begin{equation}\label{e312}
 H_k(z,w) = \lambda v_{\sigma(k)}  z_{\sigma(k)}^{p_{\sigma(k)}/q_k}/\delta(z,w)^{1/q_k}, \quad \text{for} \quad k \in K.
\end{equation}
where $v_{\sigma(k)}^{q_k} = u_{\sigma(k)k}$.

To show that $\#\sigma(K) \ge n-s$, we shall need the following lemma, whose proof is deferred to Section \ref{s:geom}.
\begin{lemma}\label{finitemul}
Let $H\colon (\bC^{n+1},0) \to (\bC^{N+1},0)$ be a holomorphic mapping sending $P^n_p$ to $P^N_q$, transversal to $P^N_{q}$ at $0$. Then $H$ is finite at $0$, i.e.\ the ideal $\mathcal I(H)$ generated by the components of $H$ has finite codimension in $\bC\{z,w\}$.
\end{lemma}

We observe from \eqref{g} that $H$ is transversal to $P^N_{q}$ at $0$, which is well known to be equivalent to $w\in \mathcal I(H)$ (see e.g.\ \cite{ER06} for a general discussion). Thus, by Lemma \ref{finitemul}, $\mathcal I(H)$ has finite codimension in $\bC\{z,w\}$. Since $w\in \mathcal I(H)$, we also have
$$
\dim \bC\{z,w\}/\mathcal (H)=\dim \bC\{z\}/\mathcal I(h),
$$
where $h(z):=H(z,0)$.
Since $h_k\equiv 0$ for $j\in \{s+1,\dots N\} \setminus K$, $h_{N+1}\equiv 0$, and $h_k(z)$, for $k\in K$, differs from $z_{\sigma(k)}^{p_{\sigma(k)}/q_k}$ only by a unit in $\bC\{z\}$, we conclude that $\mathcal I(h)$ is contained in the ideal generated by $h_1,\ldots, h_s$ and $z_{\sigma(k)}$ for $k\in K$. Since the latter ideal must have finite codimension, it follows that $s+\#\sigma(K)\geq n$, or equivalently, $\#\sigma(K)\geq n-s$. This proves the existence of the subset $K\subset \{s+1,\ldots, N\}$ and the mapping $\sigma\colon K\to \{1,\ldots, n\}$ possessing the properties claimed in (A) of Theorem \ref{explicitformula}. The existence of $r\in \bR$, $\lambda>0$, $b\in \bC^N$ and a unitary matrix $U=(u_{ij})$ was established above. The property in (B) of $U$ and that in (C) of $b$ were also established.

If $s=0$, then we have $bU=0$, and since $U$ is invertible, we deduce that $b=0$. If $N=n$, then (A) immediately implies that $K=\{s+1,\ldots,n\}$ and $\#\sigma(K)=n-s$, which in turn implies that $\sigma$ is injective. It is clear that $\sigma$ can be extended to a permutation.
\end{proof}

\begin{proof}[Proof of Theorem $\ref{Cor:alg}$]
It follows immediately from Theorem \ref{explicitformula} that $H$ is algebraic; in fact, $H^{q_k}$ for $k=1,\ldots, N$ and $H_{N+1}$ are all rational with poles along the complex hypersurface $\delta(z,w)=0$ (unless $b=0$ and $r=0$, in which case $H$ is a polynomial mapping). To complete the proof of Theorem $\ref{Cor:alg}$ it remains to verify that $\delta(z,w)$ does not vanish along $P^n_p$ (i.e.\ for $w=u+i\langle z^p,\bar z^p\rangle$). This is straightforward and left to the reader.
\end{proof}

\section{Stability group of $P_q^N$ at the origin}\label{s:auto}

A special case of Theorem~\ref{explicitformula} is when $P_p^n = P_q^N$. In this case, Theorem~\ref{explicitformula} describes the stability group of $P^N_q$ at $0$ (i.e.\ the group of automorphisms of $P_q^N$ preserving the origin), denoted by $\aut(P^N_q,0)$. This description is previously known due to work mentioned in the introduction. In this section, however, we shall (for the reader's convenience) use the formulae in Theorem~\ref{explicitformula} to provide a decomposition of the automorphisms of $P_q^N$ fixing the origin into simpler ones. First, we note that when $P^n_p=P^N_q$ the subset $K$ in (A) must equal $\{s+1,\ldots, N\}$ and $\sigma$ is a permutation of $K$ such that $q_{\sigma(k)}=q_k$ for all $k\in K$. Also, the unitary matrix $U$ in (B) must have the block form
\begin{equation}\label{equidim2}
W=
\begin{pmatrix} \tilde U & 0\\ 0& E
\end{pmatrix},
\end{equation}
where $\tilde U$ is a unitary $s\times s$ matrix and $E$ is a unitary $(n-s)\times (n-s)$ matrix (such that after reordering the coordinates $(z_{s+1},\ldots, z_{N})$ on the source side according to the permutation $\sigma$, the matrix $E$ becomes diagonal with diagonal elements of modulus one). It then also follows that $b\in \bC^N$ in (C) is of the form $b=(\beta,0)\in \bC^s\times\bC^{N-s}$. (Recall that if $s=0$, then $b=0$.)
For each permutation $\sigma$ of $K=\{s+1,\dots N\}$ such that $q_{\sigma(k)} = q_k$ for all $k\in K$,
we define
\begin{equation}
 \Sigma_{\sigma}(z,w) = (z_1,\dots ,z_s, z_{\sigma(s+1)},\dots z_{\sigma(N)}, w).
\end{equation}
Also, for each $\lambda>0$, we define the (non-isotropic) dilation $\Delta_\lambda$,
\begin{equation}\label{dilation}
 \Delta_\lambda(z,w) = (\lambda z_1,\dots, \lambda z_s,\lambda^{1/q_{s+1}}z_{s+1} \dots , \lambda^{1/q_N}z_N, \lambda^2w),
\end{equation}
and, for each $b=(\beta,0)\in \bC^s\times\bC^{N-s}$ (with $b=0$ if $s=0$) and $r>0$, we define
\begin{equation}
\Psi_{b,r}(z,w) = \left(\frac{z_1+\beta_1w}{\delta(z,w)},\ldots, \frac{z_s+\beta_sw}{\delta(z,w)}, \frac{z_{s+1}}{\delta(z,w)^{1/q_{s+1}}},\ldots, \frac{z_{N}}{\delta(z,w)^{1/q_{N}}}, \frac{w}{\delta(z,w)}\right),
\end{equation}
where as above $\delta (z,w)= 1-2i\langle z^q, b\rangle - (r+i\langle b,\bar b \rangle)w$.
Finally, for each unitary $s\times s$ matrix $\tilde U$ and $\theta_{s+1},\ldots,\theta_N\in \bR$, we define
\begin{equation}
 \Lambda_{\tilde U,\theta}(z,w) = ((z_1,\dots z_s)\tilde U, e^{i\theta_{s+1}}z_{s+1},\dots, e^{i\theta_N}z_N, w).
\end{equation}
It is readily seen from Theorem~\ref{explicitformula} that $\Sigma_{\sigma},\Delta_\lambda, \psi_{b,r}, \Lambda_{\tilde U,\theta} \in \aut(P_q^N,0)$, and it is straightforward (and left to the reader) to check, using Theorem \ref{explicitformula}, that these elementary mappings generate $\aut(P_q^N,0)$ via compositions; we mention here that $\aut(P_q^N,0)$ is a finite dimensional Lie group (see \cite{BMR02}).

\begin{theorem}\label{thm:decomposition} The stability group of $P_q^N$
at $0$ consists of mapping of the form
\begin{equation}\label{decomposition}
 T = \Delta_\lambda\circ \Lambda_{\tilde U,\theta}\circ\Psi_{b,r}\circ\Sigma_\sigma,
\end{equation}
for $\tilde U,\theta,b,r,\lambda,\sigma$ as described above.
Furthermore, the identity component $\aut_{\mathrm{Id}}(P_q^N,0)$
consists of mapping of form~\eqref{decomposition} in which $\sigma = \mathrm{Id}$.
In fact, each choice of $\sigma$ gives rise to a connected component of $\aut(P_q^N,0)$.
\end{theorem}

\begin{remark}{\rm A similar decomposition for CR mappings between connected pieces of generalized pseudoellipsoids was also given is the recent paper \cite{MontiMorbidelli12} using a very different method.
}
\end{remark}

We shall now give a proof of Theorem \ref{Main1}.

\begin{proof}[Proof of Theorem $\ref{Main1}$] The implication (ii)$\implies$(i) follows from Theorem~\ref{explicitformula} (A). Moreover, any mapping $H$ as in (ii) is of the form described in Theorem~\ref{explicitformula}. It is straightforward (and left to the reader) to check that there are $\sigma$ and $W$ as in Theorem \ref{Main1} and $\tilde U,\theta, b, r,\lambda$ as in Theorem \ref{thm:decomposition} such that
$$
H=\Delta_\lambda\circ\Lambda_{\tilde U,\theta}\circ\Psi_{b,r}\circ H_{W,\sigma}.
$$

Next, assume that (i) holds. We shall construct a transversal map $H:P_p^{n} \to P_q^{N}$ as follows. If $K=\emptyset$, then necessarily $s\ge n$. In this case, we can simply take
\begin{equation}
H(z,w) = (z_1^{p_1}, \dots z_n^{p_n}, 0,\dots 0, w).
\end{equation}
Suppose now that $K$ is nonempty. Since \(\#\sigma(K) \ge n-s\),
we can write
\[\{1,2,\dots, n\} \setminus \sigma(K) = \{t_1,t_2,\dots t_{r}\}\]
for some $r \le s$ (with $r=0$ if $\sigma(K)=\{1,\ldots, n\}$). We define a transversal (to $P_q^{N}$ at 0) map $H(z,w)=(F(z),w)$ with
\begin{equation}\label{themap}
F_k(z)=\begin{cases}
 z_{t_k}^{p_{t_k}}, \quad &k =1,2,\dots r,\\
v_kz_{\sigma(k)}^{ p_{\sigma(k)}/q_k}, &k\in K\\
0 & \text{{\rm otherwise}},
\end{cases}
\end{equation}
where the coefficients $v_k$ are chosen such that, for every $l\in \sigma(K)$,
\begin{equation}
\sum_{k\in \sigma^{-1}(l)} |v_k|^{2q_l} = 1.
\end{equation}
It is easy to check that $H$ sends $P_p^n$ into $P_q^N$. The proof is complete.
\end{proof}

\section{Proof of Lemma~\ref{finitemul}}\label{s:geom}

Recall (see \cite{BER99a}) that given a real-analytic hypersurface $M\subset \bC^{n+1}$  and $p\in M$, there are so-called normal coordinates $(z,w)\in \bC^{n}\times\bC$, vanshing at $p$, such that $M$ is given by
\[
\im w = \phi(z,\bar z, \re w),
\]
where $\phi(z,0,s)=\phi(0,\chi,s)\equiv 0$, or in complex form
\begin{equation}\label{complexdef}
w = Q(z,\bar z, \bar w),
\end{equation}
where $Q$ satisfies $Q(0,\chi, \tau) \equiv Q(z,0,\tau) \equiv \tau$ and
the reality condition
\begin{equation}
Q(z,\bar z, \bar Q(\bar z, z,w )) \equiv w.
\end{equation}
We note that $P^n_p$ and $P^N_q$ are already presented in normal coordinates.
It is convenient to use the complex defining equation \eqref{complexdef}
to define the notions of essential finiteness and essential type
as follows. We replace $\bar z, \bar w$ by independent variables $\chi, \tau$
and write
\[
Q(z,\chi, 0) = \sum_{I\in \mathbb{N}^n} q_{I}(z)\chi^I.
\]
Let $\mathcal{I}_M$ be the ideal in $\bC[[z]]$ generated by $\{q_I(z)\}_{I\in \mathbb{N}^n}$.
Following Baouendi, Jacobowitz and Treves (see \cite{BER99a}),  we shall say that $M$ is {\it essentially finite} at $p$ if $\mathcal{I}_M$ is of finite codimension
in $\bC[[z]]$. The dimension $\dim_{\bC} \bC[[z]]/{\mathcal{I}_M}$ is a biholomorphic
invariant of $M$ and is called the {\it essential type} of $M$ at $p$, denoted
by $\esstype_pM$. We note that e.g.\ $\mathcal I_{P^n_p}$ is generated by $z_1^{p_1}, \ldots, z_n^{p_n}$ and therefore $P^n_p$ is essentially finite at $0$.
Recall also that a germ of a holomorphic mapping $H\colon (\bC^{n+1},p)\to (\bC^{N+1},p')$ is said to be {\it finite} at $p$ if the ideal $\mathcal{I}(H)$ generated by the components of $H$ in the ring $\mathcal O_p$ of germs of holomorphic functions at $p$ is of finite codimension. In this case, we shall refer to this codimension as the {\it multiplicity} of $H$ at $p$,
$$
\mult_p H:=\dim_{\bC} \mathcal O_p/\mathcal{I}(H).
$$
It is well known (see e.g.\ \cite{AGV85}) that if $H$ is finite at $p$, then for every $q$ close $p$ the number of preimages $m:=H^{-1}(H(q))$ is finite and $m\leq \mult_p H$. (In the equidimensional case $N=n$, the generic number of preimages equals $\mult_p H$, but in general we only have the inequality).  Now we can state and prove the following result, which in view of the above comments regarding $P^n_p$ proves Lemma \ref{finitemul}.
\begin{proposition}\label{multid}
Let $M$ and $M'$ be real-analytic hypersurfaces in
$\bC^{n+1}$ and $\bC^{N+1}$ respectively and let $p\in M$, $p'\in M'$. Suppose that
$H:(\bC^{n+1},p) \to (\bC^{N+1},p')$ is a germ of holomorphic mapping sending
$(M,p)$ into $M'$. If $M$ is essentially finite at $p$ and $H$ is transversal
to $M'$ at $p'=H(p)$, then $H$ is finite and
\begin{equation}\label{esstype:relation}
\mult_p H \leq \esstype_p M.
\end{equation}
\end{proposition}
\begin{proof}
Suppose that $M$ and $M'$ are given in normal coordinates
$Z=(z,w)$ and $Z'=(z',w')$, vanishing at $p$ and $p'$ respectively, by complex defining functions $\rho$ and $\rho'$ of the forms:
\[
\rho(z,w,\bar z,\bar w) = w-Q(z,\bar z,\bar w), \quad \rho'(z',w',\bar z',\bar w') = w'-Q'(z',\bar z', \bar w').
\]
Since $H$ sends $M$ into $M'$, the following
holds for some real-analytic function $a(Z,\xi)$.
\begin{equation}\label{HsendsMintoMprime}
  G(Z) - Q'(F(Z), \bar H(\xi)) = a(Z,\xi) \,(w - Q(z,\xi)).
\end{equation}
Here, $H=(F,G)$ with $F=(F_1,\ldots, F_N)$. By setting $\xi = 0$, taking into account that $\bar H(0) = 0$, $Q(z,0,0) \equiv 0$ and
$Q'(z',0,0) \equiv 0$ we deduce that
\begin{equation}\label{Gw}
G(Z) = a(Z,0)\, w.
\end{equation}
Setting $w = \tau = 0$ and observing from \eqref{Gw} that $G(z,0) \equiv 0$ and $\bar G(\chi,0) \equiv 0$, we get
\begin{equation}\label{QQprime:rel}
 Q'(F(z,0), \bar F(\chi,0), 0) = a(z,0,\chi,0) \cdot Q(z,\chi,0).
\end{equation}
Since $H$ is transversal, we have $a(0) \ne 0$ (see e.g.\ \cite {BER07}). Therefore, $a(z,0,\chi,0)$ is non-vanishing for $(z,\chi)$ close to zero and hence
\begin{equation}\label{QQprime:rel2}
 a(z,0,\chi,0) ^{-1}\cdot Q'(F(z,0), \bar F(\chi,0), 0) = Q(z,\chi,0).
\end{equation}
We expand
\begin{equation}\label{expandQ}
Q(z,\chi,0) = \sum_I q_I(z)\,\chi^I.
\end{equation}
Let $\mathcal{I}_M$ and $\mathcal{I}(F)$ be the ideals in $\bC[[z]]$ generated by $\{q_I(z)\colon I\in \mathbb N^n\}$ and $\{F_j(z,0)\colon j=1,\dots N\}$, respectively. We claim that
\begin{equation}\label{inclusionideals}
\mathcal{I}_M \subset \mathcal{I}(F).
\end{equation}
Indeed, for each multi-index $I\in \mathbb{N}^n$, one has from \eqref{QQprime:rel2} that
\begin{equation}\label{diff}
q_I(z) = \frac{1}{I!}\,\frac{\partial^{I}}{\partial \chi^{I}}\biggl(a(z,0,\chi,0) ^{-1}\cdot Q'(F(z,0), \bar F(\chi,0), 0) \biggr) \biggr|_{\chi = 0}.
\end{equation}
If we expand
\begin{equation}\label{expandQprime}
Q'(z',\chi',0) = \sum_J q'_J(z)\,(\chi')^J,
\end{equation}
then it is clear from \eqref{diff} that $q_I(z)$ belongs to the ideal generated by the $q'_J(F(z,0))$, $J\in\mathbb N^N$, which in turn belongs to the ideal $\mathcal {I}(F)$ (since the ideal $\mathcal{I}_{M'}$, generated by the $q'_J(z')$, of course is contained in the maximal ideal).  Therefore, we obtain \eqref{inclusionideals}.
Furthermore, since $M$ is essentially finite, $\mathcal{I}_M$ is of finite codimension in $\bC[[z]]$ and so is $\mathcal{I}(F)$, by \eqref{inclusionideals}, and hence $F(z,0)$ is finite. Moreover,
\begin{equation}\label{e:16}
\mult_0(F(\cdot,0)) = \dim_{\bC}\bC[[z]]/\mathcal{I}(F) \leq \dim_{\bC}\bC[[z]]/\mathcal{I}_M = \esstype_0(M).
\end{equation}
On the other hand, it follows from \eqref{Gw} and the invertibility of $a(Z,0)$ that $w\in \mathcal{I}(H)$ and, hence, $H$ is also finite and
\begin{equation}\label{e:17}
\mult_0(H) = \mult_0(F(\cdot,0)).
\end{equation}
From \eqref{e:16} and \eqref{e:17}, we obtain \eqref{esstype:relation}.
\end{proof}

\section{An Example}\label{s:ex}

As mentioned in Remark \ref{rem:intro}, there is some redundancy in general in Theorem \ref{Main1}; it may happen that $H_{\sigma,W}=T\circ H_{\sigma',W'}$ for $W\neq W'$, $\sigma\neq \sigma'$, and a suitable automorphism $T$.
In the equidimensional case, the collection of all possible maps $H_{\sigma,W}$, for a given $\sigma$, is formed by the single orbit of one such map $H_{\sigma,W_0}$ under the action of the identity component $\aut_{\mathrm{Id}}(P_q^N,0)$ of the stability group, i.e.\ any $H_{\sigma,W}$ is of the form $T\circ H_{\sigma,W_0}$ for some $T\in \aut_{\mathrm{Id}}(P^N_q,0)$. (The orbit under the action of the full stability group has, potentially, several components corresponding to different permutations $\rho$ of $\{s+1,\ldots, N\}$ such that $q_{\rho(k)}=q_k$; see Section \ref{s:auto}.)

In this section, we shall give an example illustrating (hopefully) the general principle behind why both parameters $\sigma$ and $W$ in Theorem \ref{Main1} are needed in general. In particular, for a given $\sigma$, the orbit of single $H_{\sigma,W_0}$ under the action of the stability group $\aut(P^N_q,0)$ is in general ``smaller'' (in fact, lower dimensional) than the collection of all maps $H_{\sigma,W}$.

\begin{example} Let $M \subset \bC^4$ and $M'\in \bC^6$ be the hypersurface given by
\begin{equation}\label{example2}
M=\{\im w = |z_1|^4+|z_2|^8+|z_3|^{12}\},\quad M'=\{\im w' = |z'_1|^2+|z'_2|^2 + |z'_3|^2+|z'_4|^4+|z'_5|^4\}.
\end{equation}
Thus, in this example $n=3$, $N=5$, and $s=3$. In particular, $N-n=2<4=n$ and hence Theorem \ref{Main1} applies.
\begin{enumerate}[(a)]
\item Consider $K=\{4,5\}$ and $\sigma(4)=\sigma(5)=3$, and so $\#\sigma(K)=1>0=n-s$. For $a,b,c>0$ such that $a^4+b^4+c^2=1$, consider the following $n\times N$ matrix
\begin{equation}\label{WW'}
W_{a,b,c}:=
\begin{pmatrix} 1&0&0&0&0\\0&1&0&0&0\\0&0&c&a&b\end{pmatrix},
\end{equation}
which satisfies the requirements (a) and (b) in Theorem \ref{Main1}. The corresponding mapping is of the form
$$
H_{\sigma,W_{a,b,c}}(z,w)=(z_1^2, z_2^4, cz_3^6, az_3^3,bz_3^3,w).
$$
It is easy to check, using Theorem \ref{thm:decomposition}, that the orbits of $H_{\sigma,W_{a,b,c}}$ are disjoint for distinct values of $(a,b,c)$.

\item Consider the two mappings $H_{\sigma,W}$ and $H_{\sigma',W'}$, where $K=\{4,5\}$,
$$\sigma(4)=1,\ \sigma(5)=2,\qquad \sigma'(4)=2,\ \sigma'(5)=3,$$ 
and 
\begin{equation}\label{WW'2}
W:=
\begin{pmatrix} 0&0&0&1&0\\0&0&0&0&1\\1&0&0&0&0\end{pmatrix},\quad
W':=
\begin{pmatrix} 1&0&0&0&0\\0&0&0&1&0\\0&0&0&0&1\end{pmatrix}. 
\end{equation}
The corresponding mappings are of the form
\begin{equation}
\begin{aligned}
H_{\sigma,W}(z,w) &=(z_3^6, 0, 0, z_1,z^2_2,w)\\
H_{\sigma',W'}(z,w) &=(z_1^2, 0, 0, z^2_2,z^3_3,w)
\end{aligned}
\end{equation}
Again, it is straightforward to check that the orbits of $H_{\sigma,W}$ and $H_{\sigma',W'}$ are disjoint.
\end{enumerate}
\end{example}


\def\cprime{$'$}

\end{document}